\documentclass{article}
\usepackage{amssymb,amsmath,latexsym, amsthm,enumerate,  amsfonts}
\usepackage{wrapfig, tikz}
\usetikzlibrary{matrix, calc}
\usepackage{blindtext}
\title{} \author{} \date{}
\usepackage[pdftex, pdfstartview=FitH]{hyperref} 

\numberwithin{equation}{section} 

\newtheorem{thr}{Theorem}[section]
\newtheorem{te}[thr]{Theorem}
\newtheorem{df}[thr]{Definition}
\newtheorem{lm}[thr]{Lemma}
\newtheorem{lem}[thr]{Lemma}
\newtheorem{con}[thr]{Corollary}
\newtheorem{ex}[thr]{Example}

\newtheorem{fct}[thr]{Fact}
\newtheorem{fac}[thr]{Fact}

\newcommand{\Conv}{\mathop{\rm Conv}\nolimits}

\def\upar{\!\uparrow}
\def\downar{\!\downarrow}

\newcommand{\cl}{\mathop{\rm cl}\nolimits}

\newcommand{\Dec}{\mathop{\rm Dec}}

\def\Lim{
\lim\nolimits}

\usepackage[OT2,OT1]{fontenc}
\newcommand\cyr
{
	\renewcommand\rmdefault{wncyr}
	\renewcommand\sfdefault{wncyss}
	\renewcommand\encodingdefault{OT2}
	\normalfont
	\selectfont
}
\DeclareTextFontCommand{\textcyr}{\cyr}
\def\cprime{\char"7E }

 \usepackage[paperwidth=210mm, paperheight=297mm, hmargin={20mm,20mm}, vmargin={20mm,20mm} ]{geometry}
\begin{document}
\thispagestyle{empty}
\begin{center}
{\large \bf Closed sets in a generalization of the Alexandrov cube\footnote{This research was supported by the Science Fund of the Republic of Serbia, Grant No. 7750027: Set-theoretic, model-theoretic and Ramsey-theoretic phenomena in mathematical structures: similarity and diversity–SMART
}

} \vspace*{3mm}

{\bf Miloš S. Kurilić\footnote{Department of Mathematics and Informatics, Faculty of Sciences, University of Novi Sad, Serbia, e-mail: \href{mailto:anika.njamcul@dmi.uns.ac.rs}{milos@dmi.uns.ac.rs}, ORCID iD: \href{https://orcid.org/0000-0002-9268-4205}{orcid.org/0000-0002-9268-4205}}}
and
{\bf Aleksandar Pavlovi\'c\footnote{Department of Mathematics and Informatics, Faculty of Sciences, University of @Novi Sad, Serbia, e-mail: \href{mailto:apavlovic@dmi.uns.ac.rs}{apavlovic@dmi.uns.ac.rs}, ORCID iD: \href{https://orcid.org/0000-0001-5001-6869}{orcid.org/0000-0001-5001-6869}}}
\end{center}

\begin{abstract}

We continue work on the topology obtained by the convergence $\lambda_{ls}$, which started in \cite{KuPaCZ}, and further investigated in \cite{KuPaFil19}. The main goal is to describe the closed sets  and closure operator by the family of its minimal elements, with the accent on complete Boolean algebras satisfying countable chain condition. \\

 {\it AMS Mathematics  Subject Classification $(2010)$}:
03E40, 
03E17, 
06E10, 
54A20, 
54D55 
\\[1mm] {\it Key words and phrases:} complete Boolean algebra, convergence structure, algebraic convergence,  Alexandrov cube, closed sets

\end{abstract}

\section{Introduction}
\label{ch-Ouparrow}

The main topic of this paper  will be  the topology
${\mathcal O}^\uparrow$, generated  by  the convergence
$\lambda_{ls}$, defined by $\lambda_{ls}(x)=\{\limsup x\}$. This convergence is introduced in \cite{KuPaNSJOM} as $\lambda_4$, further investigated in \cite{KuPaCZ}. Also some basic properties of topology obtained by this convergence are investigated in \cite{KuPaFil19}

The accent will be on properties of closed sets and  the closure operator. The closure of some specific sets has been
determined.  In  ccc c.B.a.'s a closed set is described by the
family of its minimal elements. Some necessary conditions   for a
subset of ${\mathbb B}$ to be the set of minimal elements of a
closed set  are isolated.


\section{\normalsize Preliminaries and notation}

Our notation is mainly standard.  $\omega$ denotes the set of
natural numbers, $Y^X$ the set of all functions $f: X \rightarrow Y$
and $\omega^{\uparrow \omega}$ the set of all strictly increasing
functions from $\omega$ into $\omega$. A {\bf sequence} in a set $X$
is each function $x:\omega \rightarrow X$; instead of $x(n)$ we
usually write $x_n$ and also $x=\langle x_n:n\in \omega\rangle $.
The {\bf constant sequence} $\langle a,a,a,\dots\rangle$ is denoted
by $\langle a \rangle$.  If $f \in \omega^{\uparrow \omega}$, the
sequence $y =x \circ f$  is said to be a {\bf subsequence} of the
sequence $x$ and we write $y \prec x$.

If $\langle X, {\mathcal O}\rangle $ is a topological space, a point
$a\in X$ is said to be a {\bf limit point of a sequence} $x\in
X^{\omega }$ (we will write: $x \rightarrow_{\mathcal O}a$) iff each
neighborhood $U$ of $a$ contains all but finitely many members of
the sequence. A space $\langle X,{\mathcal O}\rangle $ is called
{\bf sequential} iff a set $A \subset X$ is closed whenever it
contains each limit of each sequence in $A$.

If $X$ is a non-empty set, each mapping $\lambda : X^{\omega
}\rightarrow P(X)$ is a {\bf convergence} on $X$ and the mapping
$u_{\lambda }:P(X)\rightarrow P(X)$, defined by $u_{\lambda }(A)=
\bigcup_{x\in A^{\omega }}\lambda (x)$, the {\bf operator of
sequential closure} determined by $\lambda$. A convergence $\lambda$
satisfying $|\lambda(x)|\leq 1$, for each sequence $x$ is called a
{\bf Hausdorff convergence}.

If $\langle X,{\mathcal O}\rangle $ is a topological space, then the
mapping $\lim_{\mathcal O}:X^\omega\rightarrow P(X)$ defined by
$\lim_{\mathcal O}(x)=\{ a\in X :x \rightarrow_{\mathcal O}a \}$ is
{\bf the convergence on} $X$ {\bf determined by the topology}
${\mathcal O}$ and for the convergence $\lambda=\lim_{\mathcal O}$ we have (see \cite{Eng85})\\[-3mm]

(L1) $\forall a \in X\;\;  a \in \lambda (\langle a\rangle )$;

(L2) $\forall x \in X^\omega\; \forall y \prec x\; \lambda (x)
\subset \lambda (y)$;

(L3) $\forall x \in X^\omega \;
     \forall a \in X\;
     ((\forall y \prec x\;
     \exists z\prec y \;
     a \in \lambda(z)) \Rightarrow a \in \lambda(x))$.\\[-3mm]

A convergence $\lambda : X^{\omega }\rightarrow P(X)$ is called a
{\bf topological convergence} iff there is a topology ${\mathcal O}$
on $X$ such that $\lambda =\lim_{\mathcal O}$. The following fact
(see, for example, \cite{KuPaNSJOM}) shows that each convergence has
a minimal topological extension  and connects topological and
convergence structures.
\begin{fac}\rm\label{T1201}
Let $\lambda : X^{\omega }\rightarrow P(X)$ be a convergence on a
non-empty set $X$. Then

(a) There is the maximal topology ${\mathcal O}_{\lambda }$  on $X$
satisfying $\lambda  \leq \lim\nolimits_{{\mathcal O}_\lambda}$;

(b) $ {\mathcal O}_\lambda=
    \{ O\subset X : \forall x \in X^\omega \, (O \cap \lambda(x) \neq \emptyset
    \Rightarrow \exists n_0 \in \omega~ \forall n \geq n_0~ x_n \in O)\}$;

(c) $\langle X,{\mathcal O}_\lambda \rangle$ is a sequential space;

(d) ${\mathcal O}_\lambda = \{ X\setminus F : F\subset X \wedge
u_{\lambda }(F) =F \}$,
    if $\lambda$ satisfies (L1) and (L2);

(e) $\lim _{{\mathcal O}_{\lambda }} =\min \{ \lambda ' \in \Conv
(X) :
                                      \lambda ' \mbox{ is topological and } \lambda \leq \lambda ' \}$;

(f) ${\mathcal O}_{\lim _{{\mathcal O}_{\lambda }}} ={\mathcal
O}_{\lambda }$;

(g) If $\lambda _1 : X^{\omega }\rightarrow P(X)$ and $\lambda _1
\leq \lambda $,
    then ${\mathcal O} _{\lambda } \subset {\mathcal O} _{\lambda _1}$;

(h) $\lambda $ is a topological convergence iff $\lambda =\lim
_{{\mathcal O}_{\lambda }}$.
\end{fac}

In general, $\lambda$ does not have to satisfy conditions (L1)-(L3).
The minimal closures of a convergence under (L1), (L2) and (L3) are
described in the following general fact (see \cite{KuPaNSJOM}).

\begin{fac}\rm\label{T1247}
Let  $\lambda : X^{\omega }\rightarrow P(X)$ be a convergence. Then

(a) The convergence $\lambda ' : X^{\omega }\rightarrow P(X)$
defined by
$$
\lambda ' (x)= \left\{
       \begin{array}{ll}
        \lambda (x) \cup \{ a \}      &  \mbox{ if } x=\langle a \rangle ,\mbox{ for some } a \in X ,\\
        \lambda (x)                   &  \mbox{ otherwise,}
       \end{array} \right.
$$
is the minimal convergence satisfying (L1) and $\lambda \leq
\lambda'$;

(b) If $\lambda$ satisfies (L1), then the mapping  $\bar{\lambda}
:X^\omega \rightarrow P(X)$ defined by $$ \textstyle
\bar{\lambda}(y)=\bigcup _{x\in X^\omega, f\in \omega^{\uparrow
\omega}, y=x\circ f}\lambda(x) $$ is the minimal convergence bigger
than $\lambda$ and satisfying (L1) and (L2);

(c) If $\lambda$ satisfies (L1) and (L2), then $\lambda^* :X^\omega
\rightarrow P(X)$ defined by $$ \textstyle \lambda^*(y)=
\bigcap_{f\in \omega^{\uparrow \omega}}
                 \bigcup_{g\in \omega^{\uparrow \omega}}{\lambda}(y \circ f \circ g)
$$ is the minimal convergence bigger than ${\lambda }$ and satisfying
(L1)-(L3);

(d) $\lambda \leq \lambda' \leq \lambda' \bar{\enspace} \leq
\lambda' \bar{\enspace} ^* \leq \lim _{{\mathcal O}_{\lambda }}$;

(e) ${\mathcal O} _{\lambda } ={\mathcal O} _{\lambda' }={\mathcal
O} _{\lambda' \bar{\enspace} } ={\mathcal O} _{\lambda'
\bar{\enspace} ^* }$.
\end{fac}

A convergence $\lambda:X^\omega
\rightarrow P(X)$ will be called {\bf weakly-topological} iff it
satisfies conditions (L1) and (L2) and its (L3)-closure, $\lambda
^*$, is a topological convergence. More about weakly-topological convergence, including the proofs of the following general facts,  can
be found in \cite{KuPaNSJOM}.

\begin{fac}\label{T1263}\label{T1220} \rm
Let $\lambda:X^\omega \rightarrow P(X)$ be a convergence satisfying
(L1) and (L2).

(a) $\lambda$ is weakly topological iff $\lambda^*=\lim_{{\mathcal
O}_\lambda }$, that is for each $x\in X^\omega$ and $a\in X$
$$
\textstyle a \in \lim_{{\mathcal O}_\lambda}(x) \Leftrightarrow
\forall y \prec x \;\; \exists z \prec y \;\; a \in \lambda(z).
$$

(b) If $\lambda$ is a Hausdorff convergence, then $\lambda^*$ is
also a Hausdorff convergence and $\lambda^*=\lim_{{\mathcal
O}_\lambda}$, that is $\lambda$ is a weakly-topological convergence.

\end{fac}

\begin{df}\label{D2301}
For an a convergence $\lambda:X^\omega\rightarrow P(X)$
 let $u_\lambda:P(X)\rightarrow P(X)$ be
the operator defined by
$$u_\lambda(A)=\textstyle
\bigcup_{y\in A^\omega}\lambda(y).$$

When it is clear which  convergence  is used, instead of
$u_\lambda$ we will briefly write    $u$.
\end{df}

\begin{fac}\label{T1206}\rm
Let $\lambda : X^{\omega }\rightarrow P(X)$ be  a convergence
satisfying (L1) and (L2) and let the mappings $u_{\lambda}^{\alpha
}: P(X) \rightarrow P(X)$, $\alpha \leq \omega _1$, be defined by
recursion in the following way: for $A\subset X$

$u_{\lambda}^0(A)=A$,

$u_{\lambda}^{\alpha +1}(A) = u_{\lambda }(u_{\lambda}^{\alpha
}(A))$ and

$u_{\lambda}^{\gamma }(A) = \bigcup _{\alpha <\gamma
}u_{\lambda}^{\alpha }(A)$, for a limit $\gamma \leq \omega _1$.

\noindent Then $u_{\lambda}^{\omega _1}$ is the closure operator in
the space $\langle X, {\mathcal O}_\lambda \rangle $.
\end{fac}

In a Boolean algebra ${\mathbb B}$, for a set  $A\subset{\mathbb B}$
let $A\upar=\{ b \in \mathbb{B} : \exists a \in A \, a\leq b\}$ and
$A\downar=\{ b \in \mathbb{B} : \exists a \in A \, a\geq b\}.$
Instead of $\{b\}\upar$ and $\{b\}\downar$ we write $b \upar$ and
$b\downar$ respectively. Clearly, $A\upar =\bigcup_{a \in A} a\upar$
and $A\downar =\bigcup_{a \in A} a\downar$. We will say that a set
$A$ is {\bf upward closed} iff $A=A\upar$ and a set $A$ is {\bf
downward closed} iff $A=A\downar$. Downward closed sets are also
called {\bf open} sets.

By $|X|$ we denote the {\bf cardinality} (the size) of the set $X$
and, if $\kappa$ is a cardinal, then $[X]^\kappa=\{A\subset
X:|A|=\kappa\}$  and $[X]^{<\kappa}=\{A\subset X : |A|<\kappa\}$.

\begin{fct}\label{T1001}
Let $x$ be a sequence in a topological space $\langle X, {\mathcal
O}\rangle $. Then

(a)
\begin{eqnarray*}
\Lim x & = & \textstyle \bigcap \left \{F \in {\mathcal F} : x_n \in
F \mbox{
for infinitely many  } n \in \omega \right\} \\
& = &\textstyle \bigcap_{A \in [\omega]^{ \omega}}\overline{\left \{
x_n:n\in A \right \}} \\
&=&\textstyle \bigcap_{f \in \omega^{\uparrow
\omega}}\overline{\left \{ x_{f(n)}:n\in \omega \right \}}.
\end{eqnarray*}

(b) $\lim x$ is a closed set.

(c) $\{x_n: n \in \omega\} \cup \lim x \subset \overline{\{x_n: n
\in \omega\}}$, but the reversed inclusion is not true in general.

(d) For each $m \in \omega$ there holds $\Lim \langle x_n : n\in
\omega \rangle = \Lim \langle x_{m+n} : n \in \omega \rangle$.
\end{fct}

\begin{fct}\label{T2301}
If convergence $\lambda$ satisfies conditions (L1)
and (L2) then

(a) $u(\emptyset)=\emptyset$;

(b) $A \subset u(A)$;

(c) $A\subset B \Rightarrow u(A) \subset u(B)$;

(d) $u(A \cup B)=u(A) \cup u(B)$.
\end{fct}

\begin{fct} \label{T2403} Let $\lambda$ be convergence satisfying (L1) and (L2) and let $A\subset X$.
Then the following conditions are equivalent:

(a) The set $A \subset X$ is closed in the topology ${\mathcal
O}_{\lambda}$;

(b) $A=u(A)$;

(c) $\lambda$-limit of each sequence in $A$ is a subset of $A$.
\end{fct}

\subsection{Forcing}
Several results in the rest of the paper are formulated and proved  by the
method of {\bf forcing}. Roughly speaking, the
forcing construction has the following steps.
First, for a convenient complete Boolean algebra ${\mathbb B}$ belonging to the model $V$ of ZFC in which we work
(the {\bf ground model}), the class $V^{{\mathbb B}}$ of ${\mathbb B}${\bf -names} (i.e.\ special
${\mathbb B}$-valued functions) is constructed by recursion.
Second, for each ZFC formula
$\varphi(v_0,\ldots,v_n)$ and arbitrary names $\tau_0,\ldots,
\tau_n$ the {\bf Boolean value} $\|\varphi(\tau_0,\ldots,\tau_n)\|$
is defined by recursion. Finally, if $G\subset {\mathbb B}$ is a
${\mathbb B}$-{\bf generic filter over} $V$ (i.e. $G$ intersects all
dense subsets of ${\mathbb B}^+$ belonging to $V$) then for each name
$\tau$ the $G$-{\bf evaluation} of $\tau$, denoted by $\tau_G$ is
defined by $\tau_G=\{\sigma_G:\sigma\in {\rm dom}(\tau) \wedge
\tau(\sigma)\in G\}$ and $V_{\mathbb B}[G]=\{\tau_G: \tau\in
V^{\mathbb B}\}$ is the corresponding {\bf generic extension} of
$V$, the minimal model of ZFC such that $V\subset V_{\mathbb
B}[G]\ni G$. The properties of $V_{\mathbb B}[G]$ are controlled by
the choice of ${\mathbb B}$ and $G$ and by the {\bf forcing
relation} $\Vdash$ defined by
$$b\Vdash \varphi(\tau) \
\stackrel{def}{\Longleftrightarrow}\forall G \in {\mathcal
G}_V^{\mathbb B} \left(b \in G \Rightarrow V_{\mathbb B}[G]
\vDash\varphi(\tau_G)\right).$$
(Here ``$G \in {\mathcal G}_V^{\mathbb B}$" will be an abbreviation for
``$G$ is a ${\mathbb B}$-generic filter over $V$".)
If $A\in V$, then there is a ${\mathbb B}$-name
$\check{A}=\{\langle a,1\rangle :  a\in A\}$ such that $(\check{A})_G=A$, in each extension
$V_{\mathbb B}[G]$. A proof of the following
statement can be found in \cite{Jec97}.

\begin{fac}\label{T719}\rm
If $\varphi$ and $ \psi$ are ZFC formulas  and $A\in V$, then

(a) $\| \varphi \wedge \psi\|=\| \varphi\| \wedge \| \psi\|$;

(b) $ \|\neg \varphi\|=\|\varphi\|'$;

(c) $\|\forall x\, \varphi(x)\|=\bigwedge_{\tau\in V^{\mathbb B}}\|\varphi(\tau)\|$;

(d) $\|\forall x \in \check{A} \, \varphi(x) \|=\bigwedge_{a \in A}\| \varphi(\check{a}\|$;

(e) $b \Vdash \varphi$ if and only if $b \leq \|\varphi\|$;

(f) $1\Vdash \varphi \Rightarrow \psi $ if and only if $\| \varphi\|\leq \|\psi\|$;

(g) If ${\rm ZFC}\vdash \varphi(x)$, then $1 \Vdash \varphi(\tau)$, for each $\tau \in V^{\mathbb B}$;

(h) If $V_{\mathbb B}[G]\vDash \varphi $, then there is $b \in G$ such that $b \Vdash \varphi$;

(i) If $1 \Vdash \exists x \, \varphi(x)$, then $1\Vdash \varphi(\tau)$, for some $\tau\in V^{\mathbb B}$ (The Maximum
Principle).
\end{fac}

\begin{fct}\label{T1803}~

(a) $b\Vdash \phi(\tau_1, \ldots, \tau_n)\mbox { iff }b \leq \|
\phi(\tau_1, \ldots, \tau_n)\|.$

(b) $1\Vdash \phi\Rightarrow\psi\mbox { iff }\| \phi\| \leq\|
\psi\|.$

(c) $V_{\mathbb B}[G]\vDash \phi\mbox { iff }\|\phi \| \in G.$
(Forcing Theorem)

(d) $V_{\mathbb B}[G]\vDash \phi\mbox { iff }\exists p\in G\,
p\Vdash \phi.$

(e) If $1 \Vdash \exists x \, \phi(x)$, then $1\Vdash \phi(\tau)$
for some $\tau\in V^{\mathbb B}$ (The Maximum Principle).
\end{fct}

%
%

\section{Generating the topology  ${\mathcal O}^\uparrow$}

 Let $\lambda_{ls}:
{\mathbb B}^\omega  \rightarrow P({\mathbb B})$ be the convergence on a complete Boolean algebra ${\mathbb B}$ defined by

$$\lambda_{ls}(x) = \{\limsup x\}.$$

Obviously,  $\lambda_{ls}$  satisfies condition (L1).

For the sequence    $x=\langle 1,0,1,0,1,0, \ldots\rangle $  in
${\mathbb B}$ and its subsequence $y$  defined by $y=\langle
0,0,0,\ldots\rangle $  we have  $\limsup x=1$ and $\limsup y=0$. So,
$\lambda_{ls}(x)\not \subset \lambda_{ls}(y)$, hence
$\lambda_{ls}$  does not satisfy (L2).

Therefore, we close the convergence $\lambda_{ls}$ under (L2) (see
Fact \ref{T1247}).

\begin{thr}\label{T4003s}
For each $y \in \mathbb{B}^\omega$ there holds
 $$\bar{\lambda}_{ls}(y)=(\limsup y
)\uparrow=\| |\tau_y|=\check{\omega}\|\uparrow.$$
\end{thr}
\begin{proof} Let $y=\langle y_n:n\in \omega\rangle $ be a sequence
in ${\mathbb B}$. The closure of $\lambda_{ls}$ under (L2),
according to Fact \ref{T1247}, is defined by
$$\bar{\lambda}_{ls}(y)=\bigcup\nolimits_{x\in {\mathbb B}^\omega,
f\in \omega^{\uparrow \omega}, y=x\circ f}\lambda_{ls}(x).$$

According to \cite[Lemma 2, a)]{KuPa07}, we have that $\limsup y=\| |
\tau_y| =\check{\omega}\|$. Therefore, it remains  to be proved that
$\bar{\lambda}_{ls}(y)=(\limsup y )\uparrow$.

($\subset$) Let $b \in \bar{\lambda}_{ls}(y)$. Then there exists a
sequence $x$ in ${\mathbb B}$ and a function $f\in \omega^{\uparrow
\omega}$ such that $y = x\circ f$ and $b \in \lambda_{ls}(x)$.
Therefore  $b=\limsup x$, and, from \cite[Lemma 2, a)]{KuPa07}, it follows
$\limsup y \leq \limsup x =b$, so $b \in (\limsup y) \uparrow$.

($\supset$) Let $x$ be a sequence in ${\mathbb B}$ and $b\geq
\limsup x$. Let us prove that $b \in \bar{\lambda}_{ls}(x)$. Let us
define the sequence $y$ in ${\mathbb B}$  by $y=\langle
x_0,b,x_1,b,x_2,\ldots\rangle$. Then

$$\limsup y=\textstyle  \bigwedge_{k\in \omega} \bigvee_{n \geq k} y_n
=\textstyle  \bigwedge_{k\in \omega}(b \vee \bigvee_{n \geq\frac k2
} x_{n}) =\textstyle  b \vee \bigwedge_{k\in \omega}\bigvee_{n \geq
\frac k2} x_{n} =$$ $$= \textstyle  b \vee \bigwedge_{k\in
\omega}\bigvee_{n \geq k} x_{n} = b \vee \limsup x = b.$$

For $f(k)=2k$ there holds $x=y \circ f$. Hence $b\in
\bar{\lambda}_{ls}(x)$. \end{proof}


Using this result  as a motivation, instead of
 $\bar{\lambda}_{ls}$ we will  write
 $\lambda^\uparrow$. According to Fact
\ref{T1247}, the topology generated by the convergence $\lambda_{ls}$
coincides with the topology generated by the convergence
${\lambda}^\uparrow$. Further on, this topology will be denoted by
 ${\mathcal O}^\uparrow$ and
the corresponding family of closed sets by ${\mathcal F}^\uparrow$.

\subsection{Topological convergence in $\langle \mathbb{B}, {\mathcal
O}^\uparrow\rangle $}

\label{S41}

Using the convergence $\lambda^\uparrow$ defined on a
c.B.a.\ ${\mathbb B}$ we have obtained the  topological space
$\langle {\mathbb B}, {\mathcal O}^\uparrow\rangle $ with the a
posteriori limit $\lim_{{\mathcal O}^\uparrow}$, which will be in
this section  briefly denoted by $\Lim$. In this section we will
investigate it on an arbitrary complete Boolean algebra.

\begin{thr}\label{T4113}
In the space $\langle {\mathbb B}, {\mathcal O}^{\uparrow}\rangle $
we have
\begin{eqnarray*} {\rm (a)}~
\Lim x & = &\textstyle  \bigcap \left \{F \in {\mathcal F}^\uparrow
: x_n \in F \mbox{
for infinitely many  } n \in \omega \right\} ~~~~~~\\
& = &\textstyle \bigcap_{f \in \omega^{\uparrow \omega}}
\cl_{\omega_1}\left( \left \{
x_{f(n)}:n\in \omega \right \} \right )\\
&=&\textstyle \bigcap_{A \in [\omega]^{ \omega}}
\cl_{\omega_1}\left( \left \{ x_n:n\in A \right \} \right)
\end{eqnarray*}

(b) $(\limsup x)\uparrow \subset \Lim x \subset \cl_{\omega_1}\left(
\left \{ x_n:n\in \omega \right \} \right)$
\end{thr}
\begin{proof} It follows directly from  Fact \ref{T1001} (a) and (b).
\end{proof}

\begin{con}\label{T4113a} The  limit  $\lim$ of a sequence is an
upward closed set, i.e. $\lim x =\lim x \uparrow$.
\end{con}
\begin{proof} According to Fact \ref{T1001} (b), the a posteriori
limit of a sequence is always a closed set, and, by Theorem
\ref{T4105}, it is upward closed. \end{proof}

\begin{ex}\label{EX4103}For a constant sequence $\langle b: n \in \omega\rangle
$ we  have $\Lim \langle b\rangle =b\uparrow$. If $x=\langle
a,b,a,b,\ldots\rangle $, then $\Lim x=(a \vee b )\uparrow$. This
will follow from the following, more general, theorem.
\end{ex}

If $x\in {\mathbb B}^\omega $, then $\tau
_x= \{ \langle \check{n}, x_n \rangle : n\in \omega \}$ is the
corresponding ${\mathbb B}$-name for a subset of $\omega$ and, by
Lemmas 2 and 6 of \cite{KuPa07},
\begin{eqnarray*}
\liminf x & = & \| \check{\omega} \subset ^* \tau_x \| ;\\
\limsup x & = & \| |\tau_x |=\check{\omega}\| ;\\
a_x       & = & \|\forall A\in(([\omega]^\omega)^V)^{\check{~}}\, \exists B\in(([A]^\omega)^V)^{\check{~}} \,
                B\subset^*\tau_x \| ;\\
b_x       & = & \|\exists A\in(([\omega]^\omega)^V)^{\check{~}}\, \forall B\in(([A]^\omega)^V)^{\check{~}} \,
                |\tau_x\cap B|=\check{\omega}\| .
\end{eqnarray*}
\begin{lem}\rm\label{T1264}
Let ${\mathbb B}$ be a complete Boolean algebra and $x$ a sequence in ${\mathbb B}$. Then

(a) $\liminf x \leq a_x \leq b_x \leq \limsup x$;

(b) If ${\mathbb B}$ is $(\omega , 2)$-distributive, then
    $a_x= \liminf x$ and $b_x=\limsup x$;

(c) $b_x=\bigvee_{y\prec x}\bigwedge_{z\prec y} \bigvee_{m \in \omega}z_m$.
\end{lem}

\begin{thr}\label{T4114}
Let $x=\langle x_n:n \in \omega\rangle $ be a sequence with finite
range, i.e. the set $\{x_n: n \in \omega\}$ is finite and $C=\{b:
x_n=b \mbox{ for infinitely many  } n \in \omega\}$.  Then
$$\Lim \langle x_n \rangle  =(\bigvee C)\uparrow$$
\end{thr}
\begin{proof} Let $a=\bigvee C$. For each $c \in C$, the set $c
\uparrow$ is closed and contains infinitely many members of the
sequence $x$. Since $\bigcap_{c \in C}(c \uparrow)=\bigvee C
\uparrow$ is a closed set and, according to Theorem \ref{T4113},
there holds  $\lim \langle x_n\rangle \subset \bigvee C \uparrow= a
\uparrow$. Let us suppose that there exists $d \in \lim \langle
x_n\rangle \setminus (a \uparrow)$. Then we have that $a \not \leq
d$, which implies that $a \wedge d' \not \leq d$. But $(a \wedge
d')\uparrow $ also contains infinitely many members of the sequence,
but does not contain $d$. Therefore, by Theorem \ref{T4113}, we have
that $d \not \in \lim\langle x_n\rangle $. \end{proof}

\begin{te}  \rm \label{T1280}
For each complete Boolean algebra ${\mathbb B}$ the following conditions are equivalent:

(a) $\lambda_{\mathrm{ls}}$ is a topological convergence;

(b) $\lambda_{\mathrm{li}}$ is a topological convergence;

(c) ${\mathbb B}$ is an $(\omega,2)$-distributive algebra;

(d) Forcing by ${\mathbb B}$ does not produce new reals.
\end{te}

\begin{lem}  \rm \label{T4116}
Let ${\mathbb B}$  be a complete Boolean algebra. Then

(a) For each $a \in {\mathbb B}$ the function
$f_a :\langle {\mathbb B},{\mathcal O_{\lambda_{\mathrm{ls}}}}\rangle \rightarrow
\langle {\mathbb B},{\mathcal O_{\lambda_{\mathrm{ls}}}}\rangle$
defined by $f_a (x)=x \wedge a$ is continuous;

(b) $\Lim_{{\mathcal O}_{\lambda_{\mathrm{ls}}}}\neq{\lambda}_{\mathrm{ls}}$ iff
there is a sequence $x$ in ${\mathbb B}$  such that
$0 \in \Lim_{{\mathcal O}_{\lambda_{\mathrm{ls}}}}(x) \setminus {\lambda}_{\mathrm{ls}}(x)$.

(c) If $x,y\in {\mathbb B}^\omega$ and $x_n \leq y_n$, for each $n \in \omega$,
then $\Lim_{{\mathcal O}_{\lambda_{\mathrm{ls}}}} (y) \subset \Lim_{{\mathcal O}_{\lambda_{\mathrm{ls}}}}(x)$.
\end{lem}

In the forthcoming theorems we will investigate connections between
a topological convergence and $\limsup$ of a sequence.

\begin{thr}\label{T4115}
If $\limsup x={0}$ then ${0}\in \Lim x$, and therefore $\lim x
={\mathbb B}$.
\end{thr}
\begin{proof} By Theorem \ref{T4113} (b), ${\mathbb
B}=0\uparrow\subset (\limsup x)\uparrow\subset \lim x\subset
{\mathbb B}$. \end{proof}

\begin{lem}\rm \label{T1233a}
Let $x=\langle x_n: n \in \omega\rangle $ be a sequence in a c.B.a.\ ${\mathbb B}$.

(a) If $x$ is a lim\,sup-stable sequence, then in the space
$\langle {\mathbb B},{\mathcal O}_{\lambda_{\mathrm{ls}}} \rangle$ we have
\begin{equation}\label{EQ1233c} \textstyle
\overline{\{x_n:n \in \omega\}}=(\limsup x)\upar \cup \;
\bigcup_{n\in \omega} x_n\upar  ;
\end{equation}

(b) If $x$ is a lim\,inf-stable sequence, then in the space
$\langle {\mathbb B},{\mathcal O}_{\lambda_{\mathrm{li}}} \rangle$ we have
\begin{equation}\label{EQ1233f} \textstyle
\overline{\{x_n:n \in \omega\}}=(\liminf x)\downar \cup \;
\bigcup_{n\in \omega} x_n\downar  .
\end{equation}
\end{lem}

\begin{lm}\label{T4117}
In the Boolean algebra $P(\omega)$ there holds $${0}\in \Lim
x\Leftrightarrow  \limsup x={0}.$$
\end{lm}
\begin{proof} By the previous theorem we have that $\limsup x={0}$
implies ${0}\in \Lim x$.

Let us prove the inverse implication. Let us suppose that there
exists a sequence $x=\langle x_n:n\in \omega\rangle $ in $P(\omega)$
such that $0 \in \lim x$ and $\limsup x=c>0$. Let $c_0\in c$. The
set ${\mathbb B}\setminus (\{c_0\} \uparrow)$ is an open set
containing 0, so there exists $n_0$, such that  $x_n \in {\mathbb
B}\setminus (\{c_0\} \uparrow)$ for each $n \geq n_0$. So, $x_n \not
\in \{c_0\} \uparrow$, i.e.\ $c_0 \not \in x_n$, for each $n \geq
n_0$, which contradicts the fact that $c_0 \in \limsup x$. \end{proof}


\begin{lm}\label{T4116}
Let ${\mathbb B}$  be a c.B.a.\ such that  $\Lim
\neq{\lambda}^\uparrow$. Then  there exists a sequence $x$  such
that ${0} \in \Lim x \setminus {\lambda}^\uparrow(x).$
\end{lm}

In the sequel we will give a
sufficient condition for a sequence to converge to a point. For this, we need the following lemma.

\begin{lm}\label{T4119}
Let $\langle x_n\rangle$  and $\langle z_n\rangle$  be two sequences
in ${\mathbb B}$ such that $x_n \leq z_n$, $n \in \omega$. Then
$\Lim z_n \subset \Lim x_n$.
\end{lm}
\begin{proof} It follows directly from Theorem \ref{T4113} (a) and
the fact that each closed set is upward closed. \end{proof}

\begin{thr}\label{T4120}
Let ${\mathbb B}$ be a c.B.a. Then, if for a sequence $x$ and an
element $a$ we have
\begin{equation}\label{EQ1-T4120}
\forall y \prec x~ \exists z \prec y ~ \limsup z \leq a,
\end{equation} then $a\in \lim x$.
\end{thr}

$$
\cite{KuPa07} \forall x \in {\mathbb B}^\omega ~ \exists y \prec x~ \forall z \prec y ~ \limsup z =\limsup y ,\eqno{( \hbar )}
$$

But, in the class of algebras with
condition ($\hbar$), statement  (\ref{EQ1-T4120}) is a
characterization of the a posteriori limit. For more about condition ($\hbar$)  check \cite{KuPa07} and \cite{KuTo}.

\begin{thr}\label{T4121}
Let ${\mathbb B}$ be a c.B.a. satisfying $(\hbar)$. Then, for each
sequence $x$ we have
$$a\in \lim x \Leftrightarrow \forall y \prec x~ \exists z \prec y ~
\limsup z \leq a.$$
\end{thr}

\begin{te}\label{T1233}\rm
If ${\mathbb B}$ is a complete Boolean algebra satisfying condition
$(\hbar)$, then $\lambda _{\mathrm{ls} }$ and $\lambda _{\mathrm{li} }$ are weakly topological convergences.
\end{te}

\begin{lem}\rm \label{T1287}
Let ${\mathbb B}$ be a complete Boolean algebra, $x=\langle x_n:n\in \omega\rangle $
a sequence in ${\mathbb B}$ and $\tau_x=\{\langle \check{n},x_n\rangle : n\in \omega\}$ the corresponding ${\mathbb B}$-name
for a real. Then

(a) If $A$ is an infinite subset of $\omega$ and $f_A:\omega
\rightarrow A$ is the corresponding increasing bijection, then
 $ \| |\tau_x\cap \check{A}|=\check{\omega} \| =\limsup x\circ f_A$.

(b) The following conditions are equivalent:

(i)  $\forall f \in \omega^{\uparrow \omega} \; \exists g \in \omega^{\uparrow \omega} \; \limsup x\circ f \circ g =0$;

(ii) $\forall y\prec x \; \exists z \prec y \; \limsup z  =0$;

(iii) $\forall A \in [\omega]^{\omega}\; \exists B \in [A]^{\omega} \; \| |\tau_x \cap \check{B}| = \check{\omega }\| =0$.
\end{lem}

\begin{te}\rm \label{T1288}
If  ${\mathbb B}$ is a complete Boolean algebra satisfying
$1\Vdash_{\mathbb B} ({\mathfrak h}^V)^{\check{~}} <{\mathfrak t}$ and cc$({\mathbb B})>2^{\mathfrak h}$,
then $\lambda_{{\mathrm{ls}}}$ is not a weakly-topological convergence on ${\mathbb B}$.
\end{te}

\section{Closed sets and the closure operator}
\label{s40}

\subsection{A characterization of closed sets and some properties}

Using the convergence
$\lambda^\uparrow$,   we define the operator
$u_{\lambda^\uparrow}:P({\mathbb B})\rightarrow P({\mathbb B})$ by
$$u_{\lambda^\uparrow}(A)=\textstyle
\bigcup_{y\in A^\omega}\lambda^\uparrow (y).$$ Obviously,
$u_{\lambda^\uparrow} =u_{\lambda^\uparrow}\uparrow$.

In this chapter, instead of $u_{\lambda^\uparrow}$ we will often use
the abbreviation  $u$. Since $\lambda^\uparrow$
satisfies conditions (L1) and (L2), by Fact \ref{T2403}, the set
$F$ is closed in the space  $\langle {\mathbb B},{\mathcal
O}^\uparrow\rangle $ iff $F=u(F)$.

\begin{thr}\label{T4105} For a set $F \subset {\mathbb B}$ the following conditions are
equivalent:

(a) $F$ is closed in the space $\langle {\mathbb B},{\mathcal
O}^\uparrow\rangle $.

(b)  $F$  is upward closed and closed with respect to the $\limsup$
operator.

(c) $F$ is upward closed and for each decreasing sequence $x$ in $F$
we have $\bigwedge_{n \in \omega} x_n \in F$ ($F$ is
$\omega_1$-closed).
\end{thr}
\begin{proof} (a)$\Rightarrow$(b) Let $F$ be a closed set. Then
$F=u(F)=u(F)\uparrow$, which implies that $F$ is upward closed.
Fact \ref{T2403} implies that $F$ is closed with respect to the convergence $\lambda^\uparrow$. Since for each sequence $x$,
$\limsup x\in\lambda^\uparrow(x)$, we conclude that $F$ is closed
with respect to the $\limsup$ operator.

(b)$\Rightarrow$(a) Let $F$ be a set closed with respect to the
$\limsup$ operator such that $F=F\uparrow$.  Let us prove that
$F=u(F)$. Let $x$ be a sequence in $F$. Then $\limsup x\in F$ and
each point $a$ such that $a\geq \limsup x$ is also in $F$. Therefore
$\limsup x \uparrow =\lambda^{\uparrow}(x)\subset F$, which implies
$u(F)\subset F$. The equality follows from Fact \ref{T2301}.

(b)$\Rightarrow$(c) It follows directly, since $\limsup$ of a
decreasing sequence is the infimum of elements of the sequence.

(c)$\Rightarrow$(b) It will be sufficient to prove that $u(F)=F$.
Let $b \in u(F)$. Then there exists a sequence $a=\langle a_n:n \in
\omega\rangle$ in $F$ such that $b \geq \limsup a$. Let $$\textstyle
c_k=b \vee \bigvee_{n\geq k} a_n, ~ k\in \omega.$$ Since $c_k \geq
a_k \in F$, we have $c_k \in F\uparrow=F$, which implies that
$c=\langle c_k:k\in \omega\rangle \in F^\omega$. It is obvious that
the sequence $c$ is decreasing, which, by (c), implies that
$\bigwedge_{k\in \omega}c_k \in F$. Finally,  from
$$\textstyle \bigwedge_{k\in \omega} c_k=\bigwedge_{k\in \omega}(b\vee \bigvee_{n\geq k}
a_n)=b \vee \bigwedge_{k\in \omega} \bigvee_{n\geq k} a_n= b \vee
\limsup a =b$$ it follows that $b \in F$. \end{proof}

\begin{con}\label{T4132}
Each closed set in the topological space $\langle {\mathbb
B},{\mathcal O}^\uparrow\rangle $  is upward closed. Each
topologically open set is a downward closed (algebraically open)
subset of the Boolean algebra ${\mathbb B}$ . Therefore, for each
point $b$ and its  neighborhood $U$  we have $b\downarrow \subset
U$. So, each not empty  open set contains 0.
\end{con}

\begin{con}\label{T4133}~

(a) The topological space $\langle {\mathbb B},{\mathcal
O}^\uparrow\rangle $ is connected.

(b) The topological space $\langle {\mathbb B},{\mathcal
O}^\uparrow\rangle $ is $T_0$ and can not be $T_1$.

(c) The topological space $\langle {\mathbb B},{\mathcal
O}^\uparrow\rangle $ is compact.
\end{con}
\begin{proof}

 (a) ${\mathbb B}$ can not be the union of two disjoint
non-empty open sets, since both of them would contain 0.

(b) Since each non-empty open set contains $0$, the space is not
$T_1$. Since, for each $b \in {\mathbb B}$, the set ${\mathbb
B}\setminus (b \uparrow)$ is open, one can easily verify that
$\langle {\mathbb B},{\mathcal O}^\uparrow\rangle $ is a $T_0$
space.

(c) The statement follows directly from the fact that the only open
set containing $1$ is ${\mathbb B}$.\end{proof}

The complementation  is not a continuous function in $\langle
{\mathbb B},{\mathcal O}^\uparrow\rangle $, since the inverse image
of an open set (which is also a downward closed set) is an upward
closed set. For meet and join as operations in ${\mathbb B}$ we have
the following.

\begin{lm}\label{T4118}
For each element $a \in {\mathbb B}$ the functions
$f_a^{\vee},f_a^{\wedge}:{\mathbb B} \rightarrow {\mathbb B}$
defined by $f_a^{\vee}(x)=x \vee a$ and $f_a^{\wedge}(x)=x \wedge a$
are continuous.
\end{lm}
\begin{proof} We will prove that $f_a^{\wedge}$ is continuous. The
proof for $f_a^{\vee}$ is analogous. Let $F$ be a closed set.   We
have
$$x\in (f_a^{\wedge})^{-1}[F] \Leftrightarrow f_a^{\wedge}(x)\in F
\Leftrightarrow x \wedge a \in F.$$ So,
$$(f_a^{\wedge})^{-1}[F]=\{x \in {\mathbb B}: x \wedge a \in F\}.$$
We will prove  that $(f_a^{\wedge})^{-1}[F]$ is closed, i.e.,  by
Theorem \ref{T4105} (c),  $(f_a^{\wedge})^{-1}[F]$ is upward closed
and $\omega_1$-closed.

Firstly, let us prove that $(f_a^{\wedge})^{-1}[F]$ is upward
closed. Let $x_1\geq x\in (f_a^{\wedge})^{-1}[F]$. Then $x_1 \wedge
a \geq x \wedge a \in F$, and since $F=F\uparrow$, we have
$x_1\wedge a \in F$, i.e.\ $x_1 \in (f_a^{\wedge})^{-1}[F]$.

It remains to be proved that $(f_a^{\wedge})^{-1}[F]$ is
$\omega_1$-closed. Let $\langle x_n :n \in \omega\rangle $ be a
decreasing sequence in $(f_a^{\wedge})^{-1}[F]$. Then $\langle
x_n\wedge a :n \in \omega\rangle $ is a decreasing sequence in $F$,
and since $F$ is $\omega_1$-closed, there holds $\bigwedge_{n\in
\omega} (x_n\wedge a)=(\bigwedge_{n\in \omega} x_n)\wedge a \in F$.
So, $\bigwedge_{n\in \omega} x_n  \in (f_a^{\wedge})^{-1}[F]$. \end{proof}

\subsection{The closure operator}

Iterating the operator $u$ $\omega_1$-times,  we obtain the operator  $\cl_{\omega_1}$ which is  the closure operator in the space $\langle
{\mathbb B},{\mathcal O}^\uparrow\rangle $.

According to Fact \ref{T2301}, we have $A \subset u(A) \subset
\cl_\alpha(A)$, for each $A \subset {\mathbb B}$ and $1\leq \alpha
\leq \omega_1$. Since the operator $u$ is defined on a Boolean
algebra, we can refine this sequence of inclusions.

\begin{thr}\label{T4102} For each set $A\subset {\mathbb B}$ and $\alpha\in [1,\omega_1]$ we have
$$A\subset A\uparrow \subset u(A)=u(A)\uparrow \subset \cl_{\alpha}(A)=\cl_{\alpha}(A)\uparrow\subset
(\bigwedge A)\uparrow.$$
\end{thr}
\begin{proof} ($A\subset A\uparrow$) Obvious.

($A\uparrow \subset u(A)$) Let $b\in  A\uparrow $. Then there exists
$a \in A$ such that $a \leq b$. Since $\lambda^\uparrow(\langle
a\rangle)= \limsup \langle a\rangle \uparrow= a \uparrow$, we have
$b \in \lambda^\uparrow(\langle a\rangle) \subset u(A)$.

($u(A)=u(A)\uparrow$) It is obvious that $u(A) \subset
u(A)\uparrow$. Let $b \in u(A)\uparrow$. Then there exists $a \in
u(A)$ such that $a \leq b$ and a sequence $x$ in $A$ such that $a
\in \limsup x \uparrow$. But then  there holds $b \in \limsup x
\uparrow$, so $b \in u(A)$.

($u(A)\uparrow \subset \cl_{\alpha}(A)$) Since, for each $\alpha
\leq \beta$, we have $\cl_{\alpha}(A)\subset \cl_{\beta}(A)$, it is
sufficient to show that $u(A)\uparrow \subset \cl_{1}(A)$, which
directly follows from
$$u(A)\uparrow=u(A)= \cl_{1}(A).$$

($\cl_{\alpha}(A)=\cl_{\alpha}(A)\uparrow$)   For an ordinal
$\alpha=\beta+1$ we have $\cl_{\alpha}(A)\uparrow=
\cl_{\beta+1}(A)\uparrow=u(\cl_{\beta}(A))\uparrow=u(\cl_{\beta}(A))=\cl_{\alpha}(A)$.

If $\alpha$ is a limit ordinal, let us suppose that for each
$\gamma<\alpha$ there holds
$\cl_{\gamma}(A)=\cl_{\gamma}(A)\uparrow$. Let $b \in
\cl_{\alpha}(A)\uparrow$. Then there exists $a \in \cl_{\alpha}(A)$
such that $a \leq b$ and there exists $\delta < \alpha$ such that $a
\in \cl_{\delta}(A)$. Since
$\cl_{\delta}(A)=\cl_{\delta}(A)\uparrow$ we have that $b \in
\cl_{\delta}(A)$, which implies that $b \in \cl_{\alpha}(A)$.

($cl_{\alpha}(A)\uparrow\subset (\bigwedge A)\uparrow$) Let $x$ be a
sequence in $A$. Then $\limsup x  \geq \bigwedge \{x_n:n\in
\omega\}\geq \bigwedge A$. Therefore $\limsup x \uparrow \subset
\bigwedge A \uparrow$, for an arbitrary sequence $x$ in $A$. This
implies that $u(A) \subset \bigwedge A \uparrow$, and then, easily,
by induction we obtain that $cl_{\omega_1}(A) \subset \bigwedge A
\uparrow$. \end{proof}

Let us find the closures of some specific sets.

\begin{lm}(The closure of a finite set)\label{T4101}
{\quad}

(a) $\cl_{\omega_1}(\{b\})=b \uparrow$, for each $b \in \mathbb{B}$.

(b) If $|A| < \aleph_0$, then $\cl_{\omega_1}(A)=\bigcup_{a\in A}
a\uparrow=A\uparrow$.
\end{lm}
\begin{proof} (a) Theorem \ref{T4102} implies $b
\uparrow=\{b\}\uparrow \subset \cl_{\omega_1}(\{b\}) \subset
\bigwedge\{b\} \uparrow=b\uparrow$.

(b) The statement follows directly  from condition (CO3) and (a).
\end{proof}


\begin{lm}\label{T4103}
{\quad}

(a) If $\langle a_n: n\in \omega\rangle $ is a decreasing sequence,
then $\cl_{\omega_1}(\{a_n:n\in \omega\})=(\bigwedge_{n \in \omega}
a_n)\uparrow$.

(b) If $\langle a_n: n\in \omega\rangle $ is an increasing sequence,
then $\cl_{\omega_1}(\{a_n:n\in \omega\})=a_0\uparrow$.

(c)  If $\{ a_n: n\in \omega\} $ is a countable antichain, then
$\cl_{\omega_1}(\{a_n:n\in \omega\})=\mathbb{B}$.

(d) If a set $A$ contains an infinite antichain, then
$\cl_{\omega_1}(A)={\mathbb B}$.

(e) If a set $A$ is dense in an infinite c.B.a.\ ${\mathbb B}$, then
$\cl_{\omega_1}(A)={\mathbb B}$.

\end{lm}
\begin{proof} ~

(a) Let $\langle a_n: n\in \omega\rangle $ be a decreasing sequence.
According to Lemma \ref{T4101} (a) we have
$\cl_{\omega_1}(\{\bigwedge_{n \in \omega}a_n\})=\{\bigwedge_{n \in
\omega}a_n\}\uparrow$, which implies that $\{\bigwedge_{n \in
\omega}a_n\}\uparrow$  is a closed set which includes the set
$\{a_n:n \in \omega\}$. So, $\cl_{\omega_1}(\{a_n:n\in
\omega\})\subset(\bigwedge_{n \in \omega} a_n)\uparrow$.  To prove
the opposite inclusion, it will be sufficient to show that
$\bigwedge_{n \in \omega} a_n \in u(\{a_n:n\in \omega\})$. This is
true, since $\limsup \langle a_n\rangle
=\bigwedge_{n \in \omega} a_n$.

(b) Let $\langle a_n :n \in \omega \rangle $ be an increasing
sequence. Then $a_0\uparrow \subset \cl_{\omega_1}(\{a_n:n\in
\omega\}) \subset \bigwedge \{a_n:n\in \omega\}\uparrow
=a_0\uparrow$.

(c) Let $\{a_n :n \in \omega\}$ be a countable antichain in
${\mathbb B}$. Without lose of generality we can suppose that  there
holds $n\neq m$ implies $a_n \neq a_m$. Let us consider the sequence
$\langle a_n:n\in \omega\rangle$. We have  $\limsup \langle a_n\rangle =0$. By Theorem \ref{T4105}, closed
sets are closed with respect to the  $\limsup$, which implies that
$0 \in \cl_{\omega_1}(\{a_n:n\in \omega\})$, and since, also by
Theorem \ref{T4105}, closed sets are upward closed, we have
$0\uparrow \subset \cl_{\omega_1}(\{a_n:n\in \omega\})$, i.e.\
${\mathbb B}=\cl_{\omega_1}(\{a_n:n\in \omega\})$.

(d) The statement is a direct consequence of  (c).

(e) The fact that each dense set in an infinite Boolean algebra
contains an infinite antichain together with (d) completes the
proof. \end{proof}

In the following examples we will show that  some inclusions from
Theorem \ref{T4102} can be strict.

\begin{ex}\label{EX4101} There exists a set $A$ such that
$cl_{\omega_1}(A)\neq (\bigwedge A)\uparrow$. Namely, let $a$ and
$b$ be two incomparable elements. Then, according to Lemma
\ref{T4101}, $\cl_{\omega_1}(\{a,b\}) =(a\uparrow) \cup (b\uparrow)
\neq (a \wedge b) \uparrow$.
\end{ex}

\begin{ex}\label{EX4102} There exists a set $A$ such that
$A\uparrow \neq u(A)$. Let $\langle a_n: n \in \omega\rangle $ be a
strictly decreasing sequence. Then we
have $\limsup \langle a_n\rangle=\bigwedge a_n$. So,
$$\bigwedge a_n=\limsup \langle a_n\rangle \in \lambda^\uparrow
(\langle a_n\rangle)\subset  u(\{a_n: n\in \omega\}),$$  but
$\bigwedge a_n \not \in \{a_n: n\in \omega\} \uparrow$.
\end{ex}




\subsection{Operator $\Dec$}

The closure operator $\cl_{\omega_1}$ is obtained by iterating the
operator $u$ $\omega_1$-times. For an arbitrary set $A$, $u(A)$ is
the union of all $\lambda^\uparrow$-limits of sequences from $A$.
But, instead of using the operator $u$ for generating the closure
operator, we can use another operator which concerns only decreasing
sequences  ($\langle y_n:n \in \omega\rangle$  is a decreasing
sequence iff $y_n \geq y_m$ for $n \leq m$).

\begin{df}\label{D4101}
Let  $\Dec:P({\mathbb B}) \rightarrow P({\mathbb B})$
 be the operator defined by
$$\Dec(A) =\textstyle \{\bigwedge_{n\in \omega}  y_n : \langle y_n\rangle \mbox { is
a decreasing sequence in }A\}.$$
\end{df}

\begin{lm}\label{T4109}
For an arbitrary $A\subset {\mathbb B}$ we have

(a) $u(A)=\Dec(A\uparrow)$;

(b) $u^2(A)=\Dec(\Dec((A\uparrow))$;

(c) $\Dec(A\uparrow) =\Dec(A\uparrow) \uparrow$.
\end{lm}
\begin{proof} {~}

(a) ($\subset$) Let $b\in u(A)$. Then there exists a sequence
$a=\langle a_n:n \in \omega\rangle$ in $A$ such that $b \geq \limsup
a$. The sequence  $\langle c_k :k \in \omega\rangle $ defined by
$$\textstyle c_k=b \vee \bigvee_{n\geq k} a_n, ~ k\in \omega$$ is a
decreasing sequence, and since $c_k\geq a_k$ we have that $\langle
c_k:k \in \omega\rangle $ is a sequence in $A\uparrow$. From
$$\textstyle \bigwedge_{k\in \omega} c_k=\bigwedge_{k\in \omega}(b\vee \bigvee_{n\geq k}
a_n)=b \vee \bigwedge_{k\in \omega} \bigvee_{n\geq k} a_n= b \vee
\limsup a =b$$ it follows that $b \in \Dec(A\uparrow)$.

($\supset$) Let  $b \in \Dec(A\uparrow)$. Then there exists a
decreasing sequence $y=\langle y_n:n\in \omega\rangle $ in
$A\uparrow$ such that $b=\bigwedge_{n\in \omega}y_n.$ For each $n
\in \omega$ we have $y_n \in A\uparrow$ and therefore there exists
$x_n \in A$ such that $y_n \geq x_n$. Hence,
$$\textstyle b=\bigwedge_{n\in \omega}y_n=\limsup y\geq \limsup x,$$
where $x=\langle x_n:n \in \omega\rangle $. This implies that $b \in
(\limsup x) \uparrow \subset u(A) \uparrow$. Since, according to
Theorem \ref{T4102}, $u(A)=u(A) \uparrow$, we have $b \in
u(A)\uparrow$.

(b) $u(u(A))=\Dec(u(A)\uparrow)=\Dec(u(A))=\Dec(\Dec(A\uparrow))$.

(c) $ \Dec(A\uparrow)=u(A)=u(A) \uparrow=\Dec(A\uparrow)\uparrow$.
\end{proof}

In the sequel we will show that iterating the operator $\Dec$
$\omega_1$-times we obtain the closure operator $\cl_{\omega_1}$.

\begin{df}\label{D4102}
Let  $\Dec_\alpha : P({\mathbb B})
\rightarrow P({\mathbb B})$, for $0<\alpha \leq \omega_1$, be
defined by

$\cdot  \Dec_1(A)=\Dec(A\uparrow)$;

$\cdot  \Dec_{\alpha+1}(A)=\Dec(\Dec_\alpha(A))$;

$\cdot  \Dec_{\gamma}(A)=\bigcup_{\alpha < \gamma} \Dec_\alpha(A)$,
for a limit ordinal $\gamma$.
\end{df}

\begin{thr}\label{T4111}
For each $\alpha \in [1, \omega_1]$ and  $A \subset {\mathbb B}$ we
have $\Dec_{\alpha}(A)=\cl_{\alpha}(A)$.
\end{thr}
\begin{proof} Let $A\subset{\mathbb B}$

For $\alpha=1$ there holds
$\Dec_1(A)=\Dec(A\uparrow)=u(A)=\cl_1(A)$. Let us suppose that for
each $\alpha \in [1,\beta)$ we have $\Dec_\alpha(A)= cl_\alpha(A)$.
Let us prove that $\Dec_\beta(A)=\cl_\beta(A)$.

If $\beta = \gamma+1$ then $\Dec_\beta(A)=\Dec_{\gamma+1}(A)=
\Dec(\Dec_{\gamma}(A))=\Dec(\cl_\gamma(A))=
\Dec(\cl_\gamma(A)\uparrow)=u(\cl(A))=\cl_{\gamma+1}(A)=\cl_\beta(A).$

If $\beta$ is a limit ordinal, then
$\Dec_\beta(A)=\bigcup_{\alpha<\beta}\Dec_\alpha(A)=
\bigcup_{\alpha<\beta}\cl_\alpha(A)=\cl_\beta(A)$. \end{proof}

Instead of iterating the operator $u$ $\omega_1$-times, Theorem
\ref{T4111} ensures us that the closure of a set $A$ can be obtained
iterating a simpler operator $\Dec$, starting with the set
$A\uparrow$.



\subsection{Closed sets and their minimal elements}

In the space $\langle {\mathbb B},{\mathcal O}^\uparrow\rangle $ a
set $F$ is closed iff $u(F)=F$. But this characterization holds in
each space obtained by a convergence satisfying (L1)
and (L2). In this subsection we will give some representations of
closed sets in the language of Boolean algebra. First we will
consider closed sets in ccc Boolean algebras.

\begin{thr}\label{T4106}
Let ${\mathbb B}$ be a  ccc c.B.a. Then each closed set $F\in
{\mathcal F}^\uparrow$ can be represented as
$$\textstyle F=\bigcup_{b \in {\rm Min}(F)}(b \uparrow),$$
where ${\rm Min}(F)=\{b \in F: b\mbox { is a minimal element of }F
\}$.
\end{thr}
\begin{proof} Let $F \subset {\mathbb B}$ such that $u(F)=F$ and let
$a \in F$. Let us suppose that there does  not exist a minimal
element of $F$ less than equal to $a$, i.e.
\begin{equation}\label{EQ1-T4106} \forall c \in F \cap a \downarrow ~ \exists d \in
F ~ d<c.
\end{equation}
Then it is clear that $0 \not \in F$. We  recursively define a
 chain
 $\langle a_\alpha:
\alpha<\omega_1\rangle $ such that $a_\alpha \in F \cap a
\downarrow$ and $a_\beta< a_\alpha$ for  $\alpha < \beta$.

Let $a_0=a$. Let us suppose that $a_\gamma$ is defined  for each
$\gamma < \beta$. If $\beta=\alpha+1$, then $a_\alpha\in F \cap a
\downarrow$ and, by (\ref{EQ1-T4106}), there exists  $a_{\alpha+1}
\in F$ such that $a_{\alpha+1}<a_\alpha$. If $\beta$ is a limit
ordinal, since $\beta < \omega_1$, then there exists an increasing
sequence of ordinals $\langle \alpha_n: n \in \omega\rangle $ such
that $\sup \alpha_n=\beta$.  Then $\langle a_{\alpha_n}:n \in
\omega\rangle $ is a strictly decreasing sequence in $F$. According
to Theorem \ref{T4105} (c), $\bigwedge_{n \in \omega}
a_{\alpha_n}\in F$. For let $a_\beta=\bigwedge_{n \in \omega}
a_{\alpha_n}$ we have $a_\beta < a_\alpha$, for each $\alpha <
\beta$.

In this way we  obtained a  chain of size $\omega_1$, which
generates the antichain $\langle a_\alpha \setminus a_{\alpha+1}
:\alpha < \omega_1\rangle$  of size $\omega_1$. This contradicts the
fact that ${\mathbb B}$ is a ccc algebra.

Since $F$ is an upward closed set, we have $F=\bigcup_{x \in F} x
\uparrow$. Also, for each $a \in F$ there exists $c \in
 {\rm Min}(F)$ such that $c\leq a$ which implies $a\uparrow \subset
c\uparrow$. So, $F=\bigcup_{a \in F} a \uparrow\subset \bigcup_{c
\in {\rm Min}(F)} c \uparrow\subset F$. \end{proof}

In general,  closed sets in $\langle {\mathbb B},{\mathcal
O}^\uparrow\rangle $ must not have minimal elements.

\begin{ex}\label{EX4112} Let ${\mathbb B}$ be a c.B.a.\ which is not
ccc. Then in ${\mathbb B}$  there exists an uncountable strictly
decreasing chain $\langle a_\alpha: \alpha < \omega_1\rangle $. Let
$A=\{a_\alpha : \alpha < \omega_1\}$. Let us show that
$A\uparrow=\bigcup_{\alpha<\omega_1} a_\alpha\! \uparrow$ is a
closed set. By Lemma \ref{T4103} (a), it will be sufficient to prove
that $\Dec(A\uparrow)=A\uparrow$. Obviously $A\uparrow \subset
\Dec(A\uparrow)$. Let $b \in \Dec(A\uparrow)$. Then there exists a
decreasing sequence $\langle b^n:n \in \omega\rangle$ in $A\uparrow$
such that $b=\bigwedge_{n\in \omega} b^n$. For each $n \in \omega$
there exists $a^n$ in $A$ such that $b^n\geq a^n$. We can w.l.o.g.\
suppose that  $\langle a^n :n \in \omega\rangle $ is also a
decreasing sequence. Since the cofinality of $\omega_1$ is not
$\omega$, there exists $\alpha< \omega_1$ such that $a_\alpha < a^n$
for each $n \in \omega$, which implies that
$$b = \bigwedge_{n\in \omega} b^n \geq \bigwedge_{n\in \omega} a^n
\geq a_\alpha.$$ So, $b\in a_\alpha\!\uparrow \subset A\uparrow$.

Since $\bigwedge A \not \in A\uparrow$ we conclude that $A\uparrow$
is a closed set without minimal elements.
\end{ex}

In Theorem \ref{T4106} it is shown that each closed set in a ccc
c.B.a.\ is determined by its minimal elements. In the following
theorem we will consider a set of the form $F=\bigcup_{x \in X} q_x
\uparrow\subset {\mathbb B}$, where $X$  is an arbitrary  set and
${\mathbb B}$  an arbitrary Boolean algebra. We will examine when
the set $\{q_x :x \in X\}$ is the set of minimal elements of $F$
and, using the set $\{q_x :x \in X\}$, we will give a
characterization of closed sets of the given form.

\begin{thr}\label{T4107}
Let ${\mathbb B}$ be a c.B.a.,  $X$  a non empty set and $\{q_x: x
\in X\} \subset {\mathbb B}$. Let  $\tau=\{\langle \check{x},q_x
\rangle : x \in X \}$ and $F=\bigcup_{x \in X} q_x \uparrow$. Then:

\noindent (1) If $x \neq y  \Rightarrow q_x \neq q_y$, then the
following conditions are equivalent:

(a) $q_x$ and $q_y$ are incomparable for different $x,y \in X$ ;

(b) $\forall x,y  \in X$ $(x \neq y \Rightarrow ||\check{x} \in \tau
\not \ni \check{y}||>0)$;

(c) $\{q_x :x \in X\}={\rm Min}(F)$.

\noindent (2) The following conditions are equivalent:

(a) $F$ is closed;

(b) $u(F)=F$;

(c) $\forall f:\omega \rightarrow X$ $\exists x \in X$ $q_x \leq
\limsup \langle q_{f(n)}\rangle $;

(d) $\forall f:\omega \rightarrow X$, where $f$ is finite to one,
$\exists x \in X$ $q_x \leq \limsup \langle q_{f(n)}\rangle $;

(e) $\forall A \in [X]^\omega$ $\exists x \in X$ $1 \Vdash \check{x}
\in \tau \Rightarrow |\tau \cap \check{A}|=\check{\omega}$.
\end{thr}
\begin{proof} (1)

(a)$\Rightarrow$(b)  Let $x,y \in X$ such that $x \neq y$. Since (a)
holds, we have $q_x \not \leq q_y$, which implies that $c=q_x \wedge
q_y'>0$. From  $c=\| \check{x}\in \tau\| \wedge \| \check{y} \in
\tau\|'=\|\check{x}\in \tau\not \ni \check{y}\|$ it follows (b).

(b)$\Rightarrow$(a) Let $x,y \in X$ such that $x \neq y$. Let
$c=\|\check{x}\in \tau\not \ni \check{y}\|$. By (b), we have that
$c=q_x \wedge q_y'>0$, so $q_x \not \leq q_y$. From $d=
\|\check{y}\in \tau\not \ni \check{x}\|>0$  it follows analogously
that $q_y \not \leq q_x$.

 (a)$\Rightarrow$(c) ($\subset$) Let $x \in X$. Then $q_x \in F$.
Let us suppose that  there exists $c \in F$ such that $c <q_x$.
Then, there exists $y \in X$ such that $q_y \leq c$, which implies
that $q_y<q_x$. A  contradiction.

($\supset$) Let $a$ be a minimal element of $F$. Then $a \in F$. So,
there exists $x \in X$ such that $q_x \leq a$. Since $q_x \in F$ and
$a $ is a minimal element, we have $a=q_x$.

(c)$\Rightarrow$(a) Let $x \neq y$. Then $q_x$ and $q_y$ are two
minimal elements, which implies that they are incomparable.

(2)

(a)$\Leftrightarrow$(b) This equivalence follows from Fact
\ref{T2403}.

 (b)$\Rightarrow$(c) Let $f:\omega\rightarrow X$. Then $\langle
q_{f(n)}:n \in \omega\rangle  $ is sequence in $F$, so,  $\limsup
\langle q_{f(n)}\rangle \in F$. Since, $F=\bigcup_{x \in X} q_x
\uparrow$, there exists $x \in X$ such that $q_x \leq \limsup
\langle q_{f(n)}\rangle $.

(c)$\Rightarrow$(d) It is obvious.

(d)$\Rightarrow$(b) Let $a=\langle a_n: n \in \omega\rangle $ be a
sequence in $F$. For each $n \in \omega$, there exists  $x_n \in X$
such that $q_{x_n} \leq a_n$. Then $\limsup \langle q_{x_n}\rangle
\leq \limsup \langle a_n\rangle $. Since $F$ is upward closed, it is
sufficient to show that $\limsup \langle q_{x_n}\rangle \in F$. Let
$f:\omega\rightarrow X$ such that $f(n)=x_n$. If $f$ is 'finite to
one', by (d), there exists $x \in X$ such that $q_x\leq \limsup
\langle q_{f(n)}\rangle = \limsup \langle q_{x_n}\rangle$.
Therefore, $\limsup \langle q_{x_n}\rangle\in F$. If $f$ is not
'finite to one', then there exists $x\in X$ such that $x_n=x$ for
infinitely many $n \in \omega$. Therefore $\limsup \langle
q_{x_n}\rangle=\bigwedge_{k \in \omega}\bigvee _{n\geq k}
q_{x_n}\geq \bigwedge_{k \in \omega} q_x=q_x$, which implies that
$\limsup \langle q_{x_n}\rangle\in F$.

(d)$\Rightarrow$(e) Let $A\in [X]^\omega$. Then there exists an
injection $f:\omega\rightarrow X$ such that $f[\omega]=A$. By (d),
there exists $x \in X$ such that
\begin{equation}\label{EQ1-T4107}
q_x \leq
\limsup \langle q_{f(n)}\rangle .
\end{equation} It is clear that $q_x=\|\check{x}\in \tau\|$.
Since $f$ is an injection, we have
\begin{eqnarray*}
\textstyle \limsup \langle q_{f(n)}\rangle &=& \textstyle
 \bigwedge_{k \in \omega} \bigvee_{n \geq k} \|f(n)\check{} \in
\tau\|\\ &=&\|\forall k \in \check{\omega}~ \exists n\geq k~
f(n)\check{} \in \tau\|\\ &=&\| |\tau \cap
\check{A}|=\check{\omega}\|.
\end{eqnarray*}From (\ref{EQ1-T4107}) we have $\|
\check{x}\in \tau\| \leq \| |\tau \cap \check{A}|=\check{\omega}\|$,
which is equivalent to $\|\check{x} \in \tau \Rightarrow |\tau \cap
\check{A}|=\check{\omega}\|=1$.

(e)$\Rightarrow$(d) Let $f: \omega\rightarrow X$ be a 'finite to
one' mapping. Then $f[\omega]\in [X]^\omega$, so, from (e) it
follows that there exists $x\in X$ such that $\|\check{x}\in
\tau\|\leq \| |\tau \cap f[\omega]\check{}|=\check{\omega}$. In
(d)$\Rightarrow$(e) part of the proof was showed that $\| |\tau \cap
f[\omega]\check{}|=\check{\omega}\|=\limsup \langle q_{f(n)}\rangle
$. This and the fact that $q_x=\|\check{x}\in \tau\|$ gives (d).
\end{proof}

In the following theorem we give a forcing characterization of
closed sets in the space $\langle {\mathbb B},{\mathcal
O}^\uparrow\rangle $, where ${\mathbb B}$ is an arbitrary c.B.a.

\begin{thr}\label{T4108}
Let $F$ be an upward closed set in ${\mathbb B}$. Then the following
conditions are equivalent:

(a) $F$ is closed;

(b) $\forall A \in[F]^\omega$ $\exists a \in F$ $a \Vdash |\Gamma
\cap \check{A}|=\check{\omega}$, where $\Gamma=\{\langle
\check{b},b\rangle :b\in {\mathbb B}^+\}$.



\end{thr}
\begin{proof}  Since $F$ is upward closed, we have $F=\bigcup_{b\in
F}b \uparrow$. By Theorem \ref{T4107} (2) (a)$\Leftrightarrow$(e),
we have  that $F$ is closed iff $\forall A \in [F]^\omega$ $\exists
a \in F$ $\| \check{a} \in \tau \Rightarrow |\tau \cap
\check{A}|=\check{\omega}\|=1$, where $\tau=\{\langle
\check{b},b\rangle: b \in F\}$.

Firstly we will prove that
\begin{equation}\label{EQ1-T4108}
1 \Vdash \tau =\Gamma \cap \check{F},\mbox { where
}\check{F}=\{\langle \check{b},1\rangle : b\in F\}.
\end{equation}
Let $b \in {\mathbb B}^+$. If $b\in F$, then $\|\check{b}\in
\tau\|=b$, and if $b \not \in F$, then $\|\check{b}\in \tau\|=0$.
Also, if $b\in F$, then $\|\check{b}\in \check{F}\|=1$, and if $b
\not \in F$, then $\|\check{b}\in \check{F}\|=0$. Therefore
$$\|\check{b}\in \tau\|=b \wedge \|\check{b}\in \check{F}\|=\|\check{b}
\in \Gamma\| \wedge\|\check{b}\in \check{F}\|=\| \check{b}\in \Gamma
\cap \check{F}\|.$$ This implies that
$$\forall b\in {\mathbb B}^+~ \|\check{b}\in \tau\|=\|\check{b} \in
\Gamma \cap \check{F}\|,$$ and by Fact \ref{T1803} (b) it follows
that
$$\forall b\in {\mathbb B}^+~ \|\check{b}\in \tau\Leftrightarrow\check{b} \in
\Gamma \cap \check{F}\|=1,$$ which is equivalent to
$$\| \forall b\in \check{{\mathbb B}}^+~  \check{b}\in \tau\Leftrightarrow\check{b} \in
\Gamma \cap \check{F}\|=1,~\mbox { i.e.}$$
$$\|\tau=\Gamma \cap \check{F}\|=1,$$
which completes the proof of (\ref{EQ1-T4108}).

Now, by (\ref{EQ1-T4108}), there holds  $$\forall A \in [F]^\omega ~
\exists a \in F~\| \check{a} \in \Gamma \cap \check{F} \Rightarrow
|\Gamma \cap \check{F} \cap \check{A}|=\check{\omega}\|=1.$$ Since,
for each ${\mathbb B}$-generic filter $G$ over $V$ we have
$\Gamma_G=G$, and $F_G=F$, then $\| \check{a} \in \Gamma \cap
\check{F} \Rightarrow |\Gamma \cap \check{F} \cap
\check{A}|=\check{\omega}\|=1$ is equivalent to the sentence that
for each generic filter $G$ there holds  $(a \in G \cap F
\Rightarrow |G\cap F \cap A|=\omega)$. Since $a \in F$ and $A
\subset F$, this is equivalent to the sentence that for each
${\mathbb B}$-generic filter $G$ there holds $(a \in G \Rightarrow
|G\cap A|=\omega)$, which is equivalent to $a \Vdash |\Gamma\cap
\check{A}|=\check{\omega}$. \end{proof}

\end{document}